\documentclass[12pt]{article}
\usepackage[utf8]{inputenc}
\usepackage[T1]{fontenc}
\usepackage{lmodern}
\usepackage{amsmath,amssymb}
\usepackage[top=2cm, bottom=1.8cm,left=2.5cm,right=2.5cm,head=2cm,headsep=0.5cm]{geometry} 
\usepackage{amsthm}
\usepackage[english]{babel}
\usepackage{graphicx}
\usepackage{dsfont}
\usepackage{url}
\usepackage{relsize}
\usepackage{color}
\usepackage{caption}
\usepackage{hyperref}
\usepackage{cleveref}
\usepackage{upgreek}
\usepackage{authblk}

\newcommand{\Hcal}{\mathcal{H}}
\newcommand{\Kcal}{\mathcal{K}}

\newcommand{\Lcal}{\mathcal{L}}

\newcommand{\Pcal}{\mathcal{P}}

\newcommand{\Scal}{\mathcal{S}}
\newcommand{\Dcal}{\mathcal{D}}

\newcommand{\Bcal}{\mathcal{B}}
\newcommand{\Ical}{\mathcal{I}}
\newcommand{\Mcal}{\mathcal{M}}
\newcommand{\Tcal}{\mathcal{T}}

\newcommand{\proj}[2]{|#1\rangle\langle #2|}
\newcommand{\eps}{\varepsilon}

\theoremstyle{plain}
\newtheorem{prop}{Proposition}[section]

\newtheorem{cor}[prop]{Corollary}

\newtheorem{thm}[prop]{Theorem}
\newtheorem{defi}[prop]{Definition}
\theoremstyle{remark}
\newtheorem{exe}[prop]{Example}
\newtheorem{rmq}[prop]{Remark}

\def\h{{\mathfrak h}}
\newcommand{\E}{\mathbb{E}}

\newcommand{\C}{\mathbb{C}}
\newcommand{\R}{\mathbb{R}}
\newcommand{\Z}{\mathbb{Z}}
\newcommand{\N}{\mathbb{N}}
\newcommand{\Pro}{\mathbb{P}}
\newcommand{\pp}{\mathbb{P}}
\newcommand{\Hb}{\mathcal{H}}

\newcommand{\Bb}{\mathcal{B}}

\newcommand{\Sb}{\mathcal{S}}
\newcommand{\Kb}{\mathcal{K}}

\newcommand{\Tr}{\mathrm{Tr}}

\newcommand{\ii}{\mathrm{i}}
\newcommand{\Lin}{\mathcal{L}}
\newcommand{\der}{\mathrm{d}}

\newcommand{\ind}{\mathds{1}}

\DeclareMathOperator{\ran}{Ran}

\newcommand{\braket}[2]{\langle #1,#2\rangle}
\newcommand{\ketbra}[2]{| #1\rangle\!\langle #2|}
\newcommand{\id}{{\mathrm{Id}}}

\begin{document}

\title{Recurrence and transience of continuous-time open quantum walks}

\author[1]{Ivan Bardet}
\author[2]{Hugo Bringuier}
\author[3]{Yan Pautrat}
\author[2]{Clément Pellegrini}

\affil[1]{Institut des Hautes \'{E}tudes Scientifiques, Universit\'{e} Paris-Saclay, 91440 Bures-sur-Yvette, France}
\affil[2]{Institut de Math\'ematiques de Toulouse; UMR5219, UPS IMT, F-31062 Toulouse Cedex 9, France}
\affil[3]{Laboratoire de Math\'ematiques d'Orsay\\
Université Paris-Sud, CNRS, Universit\'e
Paris-Saclay\\ 91405 Orsay, France}

\maketitle

\begin{abstract}
This paper is devoted to the study of continuous-time processes known as continuous-time open quantum walks (CTOQWs). A CTOQW represents the evolution of a quantum particle constrained to move on a discrete graph, but also has internal degrees of freedom modeled by a state (in the quantum mechanical sense), and contain as a special case continuous-time Markov chains on graphs. Recurrence and transience of a vertex are an important notion in the study of Markov chains, and it is known that all vertices must be of the same nature if the Markov chain is irreducible. In the present paper we address the corresponding results in the context of irreducible CTOQWs. Because of the ``quantum'' internal degrees of freedom, CTOQWs exhibit non standard behavior, and the classification of recurrence and transience properties obeys a ``trichotomy'' rather than the classical dichotomy. Essential tools in this paper are the so-called ``quantum trajectories'' which are jump stochastic differential equations which can be associated with CTOQWs.
\end{abstract}
\section*{Introduction}

Open quantum walks (OQW) have been developed originally in \cite{Att1,Att2}. They are natural quantum extensions of the classical Markov chain and, in particular, any classical discrete time Markov chain on a finite or countable set can be obtained as a particular case of OQW. Roughly speaking, OQW are random walks on a graph where, at each step, the walker jumps to the next position following a law which depends on an internal degree of freedom, the latter describing a quantum-mechanical state. From a physical point of view, OQW are simple models offering different perspective of applications (see \cite{Pet2,Pet1}). From a mathematical point of view, their properties can been studied in analogy with those of classical Markov chain. In particular, usual notions such as ergodicity, central limit theorem, irreducibility, period \cite{Att3,Pa2,Pa3,Pa4,La2,La3} have been investigated. For example, the notions of transience and recurrence have been studied in \cite{Pa1}, proper definitions of these notions have been developed in this context and the analogues of transient or recurrent points have been characterized. An interesting feature is that the internal degrees of freedom introduce a source of memory which gives rise to a specific non-Markovian behavior. Recall that, in the classical context (\cite{Norris}), an exact dichotomy exists for irreducible Markov chains: a point is either recurrent or transient, and the nature of a point can be characterized in terms of the first return time and the number of visits. In contrast, irreducible open quantum walks exhibit three possibities regarding the behavior of return time and number of visits. In this article, we study the recurrence and transience, as well as their characterizations, for continuous-time versions of OQW. 

In the same way that open quantum walks are quantum extensions of discrete-time Markov chains, there exist natural quantum extensions of continuous-time Markov processes. One can point to two different types of continuous-time evolutions with a structure akin to open quantum walks. The first (see \cite{Be1}) is a natural extension of classical Brownian motion and is called open quantum brownian motion; it is obtained by considering OQW in the limit where both time and in space are properly rescaled to continuous variables. The other type of such evolution (see \cite{Pel1}) is an analogue of continuous-time Markov chains on a graph, is obtained by rescaling time only, and is called continuous-time open quantum walks (CTOQW). In this article we shall concentrate on the latter.

Roughly speaking CTOQW represents a continuous-time evolution on a graph where a ``walker'' jumps from node to node at random times. The intensity of jumps depends on the internal degrees of freedom; the latter is modified by the jump, but also evolves continuously between jumps. In both cases the form of the intensity, as well as the evolution of the internal degrees of freedom at jump times and between them, can be justified from a quantum mechanical model.

As is well-known, in order to study a continuous-time Markov chain, it is sufficient to study the value of the process at the jump times. Indeed, the time before a jump depends exclusively on the location of the walker, and the destination of the jump is independent of that time. As a consequence, the process restricted to the sequence of jump times is a discrete time Markov chain, and all the properties of that discrete time Markov chain such as irreducibility, period,  transience, recurrence, are transferred  to the continuous-time process. This is not the case for OQW. In particular, a CTOQW restricted to its jump times is not a (discrete-time) open quantum walk. Therefore, the present study of recurrence and transience cannot be directly derived from the results in \cite{Pa1}. Nevertheless, we can still adopt a similar approach and, for instance, we study irreducibility in the sense of irreducibility for quantum dynamical systems (as defined in \cite{Dav}, see also \cite{Fa1}). As in the discrete case, we obtain a trichotomy, in the sense that CTOQW can be classified in three different possible statuses, depending on the properties of the associated return time and number of visits. 
\medskip

The paper is structured as follows. In Section 1, we recall the definition of continuous-time open quantum walks and in particular we introduce useful classical processes attached to CTOQW. Section 2 is devoted to the notion of irreducibility for CTOQW. In Section 3, we address the question of recurrence and transience and give the classification of CTOQW .

\section{Continuous time open quantum walks and their associated classical processes}\label{sectCTOQW}

This section is devoted to the introduction of continuous-time open quantum walks (CTOQW). In Subsection \ref{sectnotations}, we introduce CTOQW as a special instance of quantum Markov semigroups (QMSs) with generators preserving a certain block structure. Subsection \ref{sect:Dyson} is devoted to the exposition of the Dyson expansion associated to a QMS, which will serve as a tool in all remaining sections. It also allows us to introduce the relevant probability space. Finally, in Subsection \ref{sectQtraj} we associate to this stochastic process a Markov process called \emph{quantum trajectory} which has an additional physical interpretation, and that will be useful in its analysis.

\subsection{Definition of continuous-time open quantum walks}\label{sectnotations}
Let $V$ denotes a set of vertices, which may be finite or countably infinite. CTOQWs are quantum analogues of continuous-time Markov semigroups acting on the set $\mathrm L^{\infty}(V)$ of bounded functions on $V$. They are associated to stochastic processes evolving in the composite system
\begin{equation} \label{eq_decH1}
\Hcal=\bigoplus_{i\in V} \h_i\,,
\end{equation}
where the $\h_i$ are separable Hilbert spaces. This decomposition has the following interpretation: the label $i$ in $V$ represents the position of a particle and, when the particle is located at the vertex $i\in V$, its internal state is encoded in the space $\h_i$ (see below). Thus, in some sense, the space $\h_i$ describes the internal degrees of freedom of the particle when it is sitting at site $i\in V$. When $\h_i$ does not depend on $i$, that is if $\h_i\simeq\h$, for all $i\in V$, one has the identification $\Hcal\approx\h\otimes \ell^2(V)$ and then it is natural to write $\h_i=\h\otimes\vert i\rangle$ (we use here Dirac's notation where the \emph{ket} $|i\rangle$ represents the $i$-th vector in the canonical basis of $\ell^\infty(V)$, the \emph{bra} $\langle i |$ for the associated linear form, and $\proj{i}{j}$ for the linear map $f\mapsto \langle j|f\rangle\,|i\rangle$). We will adopt the notation $\h_i\otimes \vert i\rangle$ to denote $\h_i$ in the general case (i.e.\ when $\h_i$ depends on $i$) to emphasize the position of the particle, using the identification $\h_i\otimes\C\simeq\h_i$. We thus write:   
\begin{equation}\label{eq_decH2}
\Hcal=\bigoplus_{i\in V} \h_i\otimes \vert i\rangle\,.
\end{equation}
Last, we denote by $\mathcal I_1(\Kcal)$ the two-sided ideal of trace-class operators on a given Hilbert space $\Kcal$ and by $\Sb_{\Kcal}$ the space of density matrices on $\Kcal$, defined by:
\[\Sb_{\Kcal}=\{\rho\in\mathcal I_1(\Kcal)\mid\rho^*=\rho,\rho\geq 0 , \Tr(\rho)=1 \}.\]
A \emph{faithful} density matrix is an invertible element of $\Sb_\Kcal$, which is therefore a trace-class and positive-definite operator. Following quantum mechanical fashion, we will use the word ``state'' interchangeably with ``density matrix''.
\medskip

We recall that a quantum Markov semigroup (QMS) on $\mathcal I_1(\Kcal)$ is a semigroup $\Tcal:=(\Tcal_t)_{t\geq0}$ of completely positive maps on $\mathcal I_1(\Kcal)$ that preserve the trace. The QMS is said to be uniformly continuous if ${\lim}_{t\to 0}\|\Tcal_t-\id\|=0$ for the operator norm on $\Bcal(\Kcal)$. It is then known (see \cite{lind}) that the semigroup $(\Tcal_t)_{t\geq0}$ has a generator $\Lcal=\lim_{t\to\infty} (\Tcal_t-\id)/t$ which is a bounded operator on $\mathcal I_1(\Kcal)$, called the Lindbladian, and Lindblad's Theorem characterizes the structure of such generators. One consequently has $\Tcal_t=e^{t\Lcal}$ for all $t\geq0$, where the exponential is understood as the limit of the norm-convergent series.

Continuous-time open quantum walks are particular instances of uniformly continuous QMS on $\mathcal I_1(\Hcal)$, for which the Lindbladian has a specific form. To make this more precise, we define the following set of block-diagonal density matrices of $\Hcal$:
\[\mathcal D=\big\{\mu\in\mathcal S(\Hcal)\, ;\, \mu=\sum_{i\in V}\rho(i)\otimes\vert i\rangle\langle i\vert\big\} \, .\]
In particular, for $\mu\in\mathcal D$ with the above definition, one has $\rho(i)\in\Ical_1(\h_i)$, $\rho(i)\geq0$ and $\sum_{i\in V}\Tr\big(\rho(i)\big)=1$. In the sequel, we use the usual notations $[X,Y]=XY-YX$ and $\{X,Y\}=XY+YX$, which stand respectively for the commutator and anticommutator of two operators $X,Y\in\Bcal(\Hcal)$.

\begin{defi}\label{def_CTOQW}
Let $\Hcal$ be a Hilbert space that admits a decomposition \eqref{eq_decH1}. A \emph{continuous-time open quantum walk} is a uniformly continuous quantum Markov semigroup on $\mathcal I_1(\Hcal)$ such that its Lindbladian $\Lcal$ can be written:
\begin{equation}\label{eq_defi_CTOQW}
\begin{array}{ccll}\Lcal:&\mathcal I_1(\Hcal)&\rightarrow&\mathcal I_1(\Hcal)\\
&\mu&\mapsto &-\ii [H,\mu]+ \displaystyle{\sum_{i,j\in V}} \ind_{i\neq j} \big({S_i^j}  \mu S_i^{j*}-\frac{1}{2}\{{S_i^{j*} S_i^j},  \mu \}\big),
\end{array}
\end{equation}
 where $H$ and $(S_i^j)_{i,j}$ are bounded operators on $\Hcal$ that take the following form: 
\begin{itemize}
\item $H=\sum_{i\in V} H_i\otimes|i\rangle\langle i|$, with $H_i$ bounded self-adjoint operators on $\h_i$, i in $V$;
\item for every $i\neq j\in V$, $S_i^j$ is a bounded operator on $\Hcal$ with support included in $\h_i$ and with range included in $\h_j$, and such that the sum $\sum_{i,j\in V}\,S_i^{j*}S_i^{j}$ converges in the strong sense. Consistently with our notation, we can write $S_i^j=R_i^j\otimes\proj{j}{i}$ for bounded operators $R_i^j\in\Bcal(\h_i,\h_j)$. 
\end{itemize}
We will say that the open quantum walk is \emph{semifinite} if $\dim \h_i<\infty$ for all $i\in V$.
\end{defi}
From now on we will use the convention that $S_i^i=0$, $R_i^i=0$ for any $i\in V$. As one can immediately check, the Lindbladian $\Lcal$ of a CTOQW preserves the set $\mathcal D$. More precisely, for $\mu=\sum_{i\in V}\rho(i)\otimes\vert i\rangle\langle i\vert\in\Dcal$, denoting $\Tcal_t(\mu)=:\sum_{i\in V}\rho_{t}(i)\otimes|i\rangle\langle i| $ for all $t\geq0$, we have for all~$i\in V$
\[\frac{\der}{\der t}\rho_{t}(i)= -i[H_i,\rho_{t}(i)] +\sum\limits_{j\in V} \, \big({R_j^i}  \rho_{t}(j) R_j^{i*}-\frac{1}{2}\{{R_i^{j*} R_i^j},  \rho_{t}(i) \}\big)\,,\]

\subsection{Dyson expansion and associated probability space}\label{sect:Dyson}

In this article, our main focus is a stochastic process $(X_t)_{t\geq0}$ that informally represents the position of a particle or a walker constrained to move on $V$. In order to rigorously define this process and its associated probability space, we need to introduce the \emph{Dyson expansion} associated to a CTOQW. In particular, this allows to define a probability space on the possible trajectories of the walker. We will recall the result for general QMS as we will use it in the next section. The application to CTOQW is described shortly afterwards.
\smallskip

Let $(\Tcal_t)_{t\geq0}$ be a uniformly continuous QMS with Lindbladian $\Lcal$ on $\Ical_1(\Kcal)$ for some separable Hilbert space $\Kcal$. By virtue of Lindblad's Theorem \cite{lind}, there exists a bounded self-adjoint operator $H\in\Bcal(\Kcal)$ and bounded operators $L_i$ on $\Kcal$ ($i\in I$) such that for all $\mu\in\Ical_1(\Kcal)$,
\[\Lcal(\mu)\,=-i[H,\mu]+\displaystyle{\sum_{i\in I}}  \big(L_i\mu L_i^*-\frac{1}{2}\{L_iL_i^*,\mu\}\big)\,,\]
where $I$ is a finite or countable set and where the series is strongly convergent. The first step is to give an alternative form for the Lindbladian. First introduce 
\[ G:=-\ii H-\frac{1}{2}\sum_{i\in I} L_i^*L_i\,,\]
so that for any $\mu\in\Dcal$,
\begin{equation}\label{lindblad}
\Lcal(\mu)=G\mu+\mu G^*+\displaystyle{\sum_{i\in I}} {L_i}\,\mu\,L_i^*\,.
\end{equation}
Remark that $G+G^*+\sum_{i\in I} L_i^*L_i =0$, so that \eqref{lindblad} is the general form of the generator of a QMS, as given by Lindblad \cite{lind}. The operator $-(G+G^*)$ is positive semidefinite and $t\mapsto e^{t G}$ defines a one-parameter semigroup of contractions on $\Kcal$ by a trivial application of the Lumer-Phillips theorem (see e.g.\ Corollary 3.17 in \cite{EngelNagel}). We are now ready to give the Dyson expansion of the QMS. 

\begin{prop}\label{prop:dyson}
Let $(\Tcal_t)_{t\geq0}$ be a QMS with Lindbladian $\Lcal$ as given above. For any initial density matrix $\mu\in\Scal_\Kcal$, one has 
\begin{equation}\label{eq:prop_dyson1}
\begin{aligned}
\Tcal_t(\mu)
& =\sum_{n=0}^\infty\sum_{i_1,\ldots,i_n\in I}\int_{0<t_1<\cdots<t_n<t} \,
\,\zeta_t(\upxi)\, \mu\, \zeta_t(\upxi)^*\,\der t_{1}\cdots \der t_{n} \,,
\end{aligned}
\end{equation}
where $\zeta_t(\upxi)=e^{(t-t_n)\,G}\,L_{i_n}\,\cdots\,L_{i_1}\,e^{t_1\,G}$ for $\upxi=(i_1,\ldots,i_n;t_1,\ldots,t_n)$.
\end{prop}
We now turn to applying this to CTOQW. Due to the block decomposition of $H$ and of the $S^i_j$, one can write $G=\sum_{i\in V}G_i\otimes|i\rangle\langle i|$, where
\begin{equation}\label{eq_Gi}
  G_i=-\ii H_i-\frac{1}{2}\sum_{j\in V\setminus\{i\}} R_i^{j*}R_i^j \, .
\end{equation}
From \Cref{prop:dyson} we then get the following expression for the Lindbladian: for all $\mu=\sum_{i\in V}\rho(i)\otimes\vert i\rangle\langle i\vert$ in $\Dcal$,
\begin{equation}\label{eq_lind_alternative}
 \Lcal(\mu)=\sum\limits_{i\in V}\Big(G_i\rho(i)+\rho(i) G_i^* +\sum\limits_{j\neq i\in V}R_j^i\,\rho(j)\,R_j^{i*}\Big)\otimes|i\rangle\langle i| \, .
\end{equation}

\begin{cor}\label{prop:dyson_CTOQW}
Let $(\Tcal_t)_{t\geq0}$ be a CTOQW with Lindbladian $\Lcal$ given by \Cref{eq_lind_alternative}. For any initial density matrix $\mu\in\Dcal$, one has 
\begin{equation}\label{eq:prop_dyson_CTOQW1}
\Tcal_t(\mu)=\sum_{n=0}^\infty\sum_{i_0,\ldots,i_n\in V}\int_{0<t_1<\cdots<t_n<t} T_t(\xi)\, \rho({i_0})T_t(\xi)^*\der t_{1}\cdots \der t_{n}\,\otimes\proj{i_n}{i_n} \, ,
\end{equation}
where, for $\xi=(i_0,\ldots,i_n;t_1,\ldots,t_n)$ with $i_0,\ldots,i_n\in V^{n+1}$ and $0<t_1<\ldots<t_n<t$,
\begin{equation}\label{eq:prop_dyson_CTOQW2}
T_t(\xi)
:=e^{(t-t_{n}) G_{i_n}}\,R_{i_{n-1}}^{i_n}\,e^{(t_n-t_{n-1})G_{i_{n-1}}}\,\cdots\,e^{(t_2-t_{1})G_{i_1}}\,R_{i_0}^{i_1}\, e^{t_{1} G_{i_0}}.
\end{equation}
\end{cor}
Note the small discrepancy between $\upxi=(i_1,\ldots,i_n;t_1,\ldots,t_n)$ in \eqref{eq:prop_dyson1} and $\xi=(i_0,\ldots,i_n;t_1,\ldots,t_n)$ in \eqref{eq:prop_dyson_CTOQW1}, where the additional index $i_0$ is due to the decomposition of $\mu$.

\begin{rmq}
Equation \eqref{eq:prop_dyson1} is also called an unravelling of the QMS. It was first introduced in \cite{DaviesBook, SriDav}, with a heuristic interpretation as the average result of trajectory $\xi=(i_0,\ldots,i_n;t_1,\ldots,t_n)$ on the state $\mu$, averaged over all possible such trajectories. We discuss later connections with an operational interpretation of $T_t(\xi)\rho(i_0)T_t(\xi)^*$ in \Cref{rmk:dyson_expansion}.
\end{rmq}

The decomposition described in \Cref{eq:prop_dyson_CTOQW2} will allow us to give a rigorous definition of the probability space associated to the evolution of the particle on $V$. The goal is to introduce the probability measure $\Pro_{\mu}$ that models the law of the position of the particle, when the initial density matrix is $\mu\in \mathcal D$. The following is inspired by \cite{Bar1,Bou1,Ja1}.

First define the set of all possible trajectories up to time $t\in[0,\infty]$ as $\Xi_t:=\underset{n\in \N}\sqcup\Xi_t^{(n)}$, where $\Xi_t^{(n)}$ is the set of trajectories on $V$ up to time $t$ comprising $n$ jumps:
\[\Xi_t^{(n)}:= \{\xi=(i_0,\ldots,i_n;t_1,\ldots,t_n)\in V^{n+1}\times\R^n, \, 0<t_1<\cdots<t_n<t\} \, .\]
For $t\in\mathbb R_+$, the set $\Xi_t^{(n)}$ is equipped with the $\sigma$-algebra $\Sigma_t^{(t)}$ and with the measure $\nu_t^{(n)}$, which is induced by the map
\[\begin{array}{clcc}I_n:\, \big(V^{n+1}\times[0,t)^n, \mathcal P(V^{n+1})\times\mathcal{B}([0,t)^n),\delta^{n+1}\times\frac{1}{n!}\lambda_n \big)&\rightarrow&\big(\Xi_t^{(n)},\Sigma_t^{(n)},\nu_t^{(n)}\big)&\,,\\
(i_0,\ldots,i_n;s_1,\ldots,s_n)&\mapsto &(i_0,\ldots,i_n;s_{\min} ,\ldots,s_{\max})
\end{array}\]
where $\delta$ is the counting measure on $V$, $\mathcal{B}([0,t)^n)$ is the Borel $\sigma$-algebra on $[0,t)^n$ and $\lambda_n$ is the Lebesgue measure on $\mathcal{B}([0,t)^n)$ for all $n\geq0$. These measures are $\sigma$-finite and this allows us to apply Carathéodory's extension Theorem. We first define the $\sigma$-algebra $\Sigma_t:=\sigma(\Sigma_t^{(t)},n\in\N)$ and the measure $\nu_t$ on $\Xi_t$ such that $\nu_t=\nu_t^{(n)}$ on $\Xi_t^{(n)}$. For a given $\mu=\sum_{i\in V}\rho(i)\otimes|i\rangle\langle i| $  in~$\Dcal$,  one can then define the probability measure $\Pro_\mu^t$ on $(\Xi_t,\Sigma_t)$ such that, for all $E\in\Sigma_t$,
\begin{align}\label{eq_prob_dyson}
\Pro_\mu^t(E)
&:=\int_{E}\Tr\big( T_t(\xi)\, \mu \,T_t(\xi)^*\big)\der\nu^t(\xi)\nonumber\\
&=\sum_{n=0}^\infty\sum_{i_0,\ldots,i_n\in V}\int_{0<t_1<\cdots<t_n<t} \ind_{\xi\in E}\, \Tr\big( T_t(\xi) \rho({i_0}) T_t(\xi)^*\big)\,\der t_{1}\cdots \der t_{n} \, ,\end{align}
where $\xi=(i_0,\ldots,i_n;t_1,\ldots,t_n)$ and where $T_t(\xi)$ is defined by Equation \eqref{eq:prop_dyson_CTOQW2}. The measure $\Pro_\mu^t$ is indeed a probability measure as one can check that $\Pro_\mu^t(\Xi_t)=\Tr\big(e^{t\Lcal}(\mu)\big)=1$. The family of probability measures $\big(\Pro_{\mu}^t\big)_{t\geq 0}$ is consistent, as \Cref{eq_prob_dyson,eq:prop_dyson_CTOQW2} show that 
\[\Pro_{\mu}^{t+s}(E) =\sum_{n=0}^\infty\sum_{i_0,\ldots,i_n\in V}\int_{0<t_1<\cdots<t_n<t} \ind_{\xi\in E}\, \Tr\big(\mathrm e^{s\Lcal} \big(T_t(\xi)\, \rho({i_0})\, T_t(\xi)^* \big)\big)\,\der t_{1}\cdots \der t_{n} =\Pro_{\mu}^{t}(E)\]
for all $t,s\geq 0$ and $E\in \Sigma_t$. Hence, Kolmogorov’s consistency Theorem allows us to extend $(\Pro_{\mu}^t)_{t\geq 0}$ to a probability measure $\Pro_{\mu}$ on $(\Xi_\infty,\Sigma_\infty)$ where $\Sigma_\infty=\sigma(\Sigma_t,t\in\mathbb R_+)$.
\smallskip

In most of our discussions below we will specialize to the case where $\mu$ is of the form $\mu=\rho\otimes\proj ii$. In such a case, we denote by $\pp_{i,\rho}$ the probability $\pp_\mu$.

\subsection{Quantum trajectories associated to CTOQW}\label{sectQtraj}
\emph{Quantum trajectories} are another convenient way to describe the distribution of the position process $(X_t)_{t\geq0}$ associated to the CTOQW. Actually, the combination of quantum trajectories and of the Dyson expansion will be essential tools for the main result of this article. Formally speaking, quantum trajectories model the evolution of the state when a continuous measurement of the position of the particle is performed. The state at time $t$ can be described by a pair $(X_t,\rho_t)$ with $X_t\in V$ the position of the particle at time $t$ (as recorded by the measuring device) and $\rho_t\in\mathcal S_{\Hb}$ the  density matrix describing the internal degrees of freedom, given by the wave function collapse postulate and thus constrained to have support on $\h_i$ alone. The stochastic process $(X_t,\rho_t)_{t\geq0}$ is then a Markov process, and this will allow us to use the standard  machinery for such processes. However, their rigorous description is less straightforward than the one for discrete time OQWs. It makes use of stochastic differential equations driven by jump processes. We refer to \cite{Pel1} for the justification of the below description and for the link between discrete and continuous-time models. Remark that we denote by the same symbol the stochastic process $(X_t)_{t\geq0}$ appearing in this and the previous section. This will be justified in \Cref{rmk:dyson_expansion} below.

In order to present the stochastic differential equation satisfied by the pair $(X_t,\rho_t)$ we need a usual filtered probability space $\big(\Omega,\mathcal F,(\mathcal F_t)_t,\mathbb P\big)$, where we consider independent Poisson point processes $N^{i,j},i,j\in V,i\neq j$ on $\mathbb R^2$. These Poisson point processes will govern the jump from site $i$ to site $j$ on the graph $V$.

\begin{defi}\label{omega}
Let $(\Tcal_t)_{t\geq0}$ be a CTOQW with Lindbladian $\Lcal$ and let $\mu=\sum_{i\in V}\rho(i)\otimes|i\rangle\langle i| $ be an initial density matrix in $\Dcal$. The quantum trajectory describing the indirect measurement of the position of the CTOQW is the Markov process $(\mu_t)_{t\geq 0}$ taking values in the set $\mathcal D$ such that 
\[\mu_0=\rho_0\otimes \vert X_0\rangle\langle X_0\vert,\]
where $X_0$ and $\rho_0$ are random with distribution
\[\pp\left((X_0,\rho_0)=\left(i,\frac{\rho(i)}{\Tr(\rho(i))}\right)\right)=\Tr(\rho(i))\ \mbox{for all } i\in V\]
and such that $\mu_t=:\rho_t\otimes\vert X_t\rangle\langle X_t\vert$  satisfies for all $t\geq0$ the following stochastic differential equation:
\begin{align}\label{qtcc}
 \mu_t=&\,  \mu_0+\int_{0}^t\Mcal(\mu_{s-}) \, \der s
 +\sum_{i\neq j}\int_{0}^{t} \int_{\R}   \left(\frac{S_i^{j}\,\mu_{s-}\,S_i^{j*}}{\Tr(S_i^{j}\mu_{s-}S_i^{j*})}-\mu_{s-}\right)\ind_{0<y<\Tr(S_i^{j}\mu_{s-}S_i^{j*})} N^{i,j}(\der y, \der s)
\end{align}
where 
\[\Mcal(\mu )=\Lcal(\mu)-\sum_{i\neq j}\big(S_i^{j}\,\mu\, S_i^{j*}-\mu\, \Tr(S_i^{j}\,\mu\, S_i^{j*})\big)\]
so that for $\mu=\sum_i\rho(i)\otimes\vert i\rangle\langle i\vert\in\mathcal D$,
\begin{equation*}
\Mcal(\mu )=\sum_i\,\big(G_i\rho(i)+\rho(i)G_i^*-\rho(i)\Tr\big(G_i\rho(i)+\rho(i)G_i^*\big)\big)\otimes\vert i\rangle\langle i\vert\, .\end{equation*} 
\end{defi}
\begin{rmq}\label{rem_CTMC}
An interesting fact has been pointed out in \cite{Pel1}: continuous-time classical Markov chains can be realized within this setup by considering $\h_i\simeq\C$ for all $i\in V$. 
\end{rmq}

Let us briefly describe the evolution of the solution $(\mu_t)_t$of \eqref{qtcc}. Assume that $X_0=i_0$ for some $i_0\in V$ and consider $\rho_0$ a state on $\h_{i_0}$. We then consider the solution, for all $t\geq0$,
\begin{equation}\label{EDO}\eta_t=\rho_0+\int_0^t \big(G_{i_0}\eta_s+\eta_sG_{i_0}^*-\eta_s\Tr(G_{i_0}\eta_s+\eta_sG_{i_0}^*)\big)\,\der s\,.
\end{equation}
We stress the fact that the solution of this equation takes value in the set of states of $\h_{i_0}$ (this nontrivial fact is well known in the theory of quantum trajectories, see \cite{Pelleg2} for further details). Now let us define the first jump time. To this end we introduce for $j\neq i_0$
\[T_1^j=\inf\big\{t\geq0\,;\,N^{i_0,j}\big(\{u,y\,\vert\, 0\leq u\leq t,\, 0\leq y\leq\Tr(R_{i_0}^j\eta_u{R_{i_0}^j}^*)\}\big)\geq1\big\}.\]
The random variables $T_1^j$ are nonatomic, and mutually independent. Therefore, if we let $T_1=\inf_{j\neq i_0}\{T_1^j\}$ then there exists a unique  $j\in V$ such that $T_1^j=T_1$. In addition,
\begin{align}
\pp \big(T_1\leq\eps\big)&\leq\sum_{j\neq i_0}\pp\big(T_1^j\leq\eps\big) \nonumber\\
&=\sum_{j\neq i_0}(1-e^{-\int_0^{\eps}\Tr(R_{i_0}^j\eta_u{R_{i_0}^j}^*)du})\nonumber\\
&\leq\sum_{j\neq i_0}\int_0^\eps \Tr(R_{i_0}^j\eta_u{R_{i_0}^j}^*)du \nonumber\\
&\leq\eps\sum_{j\neq i_0}\|{R_{i_0}^j}^*R_{i_0}^j\| \label{eq_ineqPT1}
\end{align}
where the sums are over all $j$ in $V$ with $j\neq i_0$. Now remark that our assumption that $\sum_{i,j} S_i^{j*} S_i^j$ converges strongly implies that the sum $\sum_{j\neq i} \|R_i^{j*}R_i^j\|$ is finite for all $i$ in $V$,  so that \Cref{eq_ineqPT1} implies $\pp(T_1>0)=1$. On $[0,T_1]$ we then define the solution $(X_t,\rho_t)_t$ as 
\[(X_t,\rho_t)=(i_0,\eta_t)  \ \mbox{for}\ t\in[0,T_1)\quad \mbox{and}\quad
 (X_{T_1},\rho_{T_1}) = \big(j,\frac{R_i^j\eta_{T_1-}{R_i^j}^*}{\Tr(R_i^j\eta_{T_1-}{R_i^j}^*)}\big)\, .\]
We then solve 
\begin{equation}\label{EDO}
\eta_t=\rho_{T_1}+\int_0^t \big(G_j\eta_s+\eta_sG_j^*-\eta_s\Tr(G_j\eta_s+\eta_sG_j^*)\big)\,ds\, ,
\end{equation}
and define a new jumping time $T_2$ as above. By this procedure we define an increasing sequence $(T_n)_n$ of jumping times. We show that $T:=\lim_n T_n=+\infty$ almost surely: we introduce
$$ N_t=\sum_{i\neq j}\big(\int_{0}^{t\wedge T} \int_{\R} \ind_{0<y<\Tr(R_i^{j}\rho_{s-}R_i^{j*})} \,N^{i,j}(\der y, \der s)\big)$$
(where the sum is over all $i,j$ with $i\neq j$) which counts the number of jumps before $t$. In particular $N_{T_p}=p$ for all $p\in\mathbb N$. Now from the properties of the Poisson processes we have for all $p\in\mathbb N$ and all $m\in\mathbb N$,
$$\mathbb E( N_{T_p\wedge m})\leq \mathbb E\big( N_m\big)=\sum_{i\neq j}\mathbb E\big(\int_0^{m\wedge T}\Tr(R_i^{j} \rho_{s-}R_i^{j*})\der s\big)\leq m\sum_{i\neq j}\Vert R_i^{j*}R_i^j\Vert.$$
Denoting $C=\sum_{i\neq j}\Vert R_i^{j*}R_i^j\Vert$ (which is finite) the inequality $p\,\mathbb P(T_p\leq m)\leq\mathbb E(N_{T_p\wedge m})$ implies 
$$\mathbb P(T_p\leq m)\leq\frac{m}{p}C.$$
This implies that $\mathbb P(\lim_n T_n\leq m)=0$ for all $m\in\mathbb N$ so that $\lim_n T_n=+\infty$ almost surely. Therefore, the above considerations define $(X_t,\rho_t)$ for all $t\in\R_+$.

\subsection{Connection between Dyson expansion and quantum trajectories}\label{rmk:dyson_expansion}

The connection between the process $(X_t,\rho_t)$ defined in this section and the Dyson expansion has been deeply studied in the literature. We do not give all the details of this construction and instead refer to \cite{Bar1,Bou1} for a complete and rigorous justification. The main point is that the process $(X_t,\rho_t)$ defined in \Cref{sectQtraj} can be constructed explicitly on the space $(\Xi^\infty,\Sigma^\infty,\Pro)$, as we now detail.

Recall the interpretation of $\xi=(i_0,\ldots,i_n;t_1,\ldots,t_n)$ as the trajectory of a particle, initially at $i_0$ and jumping to $i_k$ at time $t_k$. First, on $(\Xi^\infty,\Sigma^\infty,\Pro)$ define the random variable $\tilde N^{i,j}_t$ by
\begin{equation*}
	\tilde N^{i,j}_t(\xi)=\mathrm{card}\big\{k=0,\ldots,n-1\,|\, (i_k,i_{k+1})=(i,j)\big\}
\end{equation*}
for $\xi=(i_0,\ldots,i_n;t_1,\ldots,t_n)$ as above. Now, let
\begin{equation} \label{eq:defXttrhott}
\begin{aligned}
	\tilde X_t(\xi)&=\left\{\begin{array}{cl}
i_k & \mbox{ if } t_k\leq t < t_{k+1}\\
i_n & \mbox{ if } t_n\leq t.\end{array}\right. \\
	\tilde \rho_t(\xi)&= \frac{T_t(\xi)\rho_0\, T_t(\xi)^*}{\Tr({T_t(\xi)\rho_{0} T_t(\xi)^*})}
\end{aligned}
\end{equation}
(recall that $T_t(\xi)$ is defined in \Cref{eq:prop_dyson_CTOQW2}) and 
\begin{equation*}
	\tilde\mu_t=\tilde\rho_t\otimes\ketbra{\tilde X_t}{\tilde X_t} \, .
\end{equation*}
Differentiating \Cref{eq:defXttrhott}, one can show that the process $(\tilde{\mu_t})_t$ satisfies
\begin{align}\label{qtcc2}
\der \tilde \mu_t=&\, \Mcal(\tilde \mu_{t-}) \ \der t
 +\sum_{i\neq j} \big(\frac{S_i^{j}\tilde \mu_{s-}S_i^{j*}}{\Tr(S_i^{j}\tilde \mu_{t-}S_i^{j*})}-\tilde \mu_{t-}\big) \,\der\tilde N^{i,j}(t).
\end{align}
It is proved in \cite{Bar1} that the processes
\[(\tilde{N}^{i,j}_t)_t\quad\mbox{and}\quad \big(\int_{0}^{t} \int_{\R} \ind_{0<y<\Tr(R_i^{j}\rho_{s-}R_i^{j*})} N^{i,j}(\der y, \der s)\big)_t\]
(for $(\rho_t)_t$ and $N^{i,j}$ defined in the previous section) have the same distribution. Therefore, $(\tilde\mu_t)_t$ and $(\mu_t)_t$ have the same distribution. For this reason, we will denote the random variables $\tilde\eta_t$, $\tilde X_t$, $\tilde\rho_t$ by $\eta_t$, $X_t$, $\rho_t$, i.e.\ we identify the random variables obtained by the construction in \Cref{sectQtraj} and those defined by \Cref{eq:defXttrhott}. In addition, from expression \eqref{eq:defXttrhott} for $\rho_t$, $X_t$ we recover immediately that $\mu_t=\rho_t\otimes \proj{X_t}{X_t}$ satisfies 
\[\ \E_{\mu_0}(\mu_t)=\Tcal_t(\mu_0)\]
where $\E_{\mu_0}$ is the expectation with respect to the probability $\pp_{\mu_0}$ defined in \Cref{sect:Dyson}. This identity shows that the quantum Markov semigroup $(\Tcal_t)_t$ plays for the process $(X_t,\rho_t)_t$ the same role as the Markov semigroup in the classical case. Because a notion of irreducibility is naturally associated to such a semigroup (see \cite{Dav,FagReb2} for general considerations on the irreducibility of Lindbladians), this will allow us to associate a notion of irreducibility to a continuous-time open quantum walk.

Now note that expressions \eqref{eq:defXttrhott} give an interpretation of $X_t$ and $\rho_t$ in terms of quantum measurement. Indeed, one can see the operator $T_t(\xi)$ for $\xi=(i_0,\ldots,i_n;t_1,\ldots,t_n)$ (or, rather, the map $\rho\mapsto T_t(\xi)\rho T_t(\xi)^*$) as describing the effect of the trajectory where jumps (up to time $t$) occur at times $t_1$,\ldots,$t_n$ and $i_0$,\ldots,$i_n$ is the sequence of updated positions: as long as the particle sits at $i_k\in V$, the evolution of its internal degrees of freedom is given by the semigroup of contraction $(e^{t\,G_{i_k}})_{t\geq0}$ and, as the particle jumps to $i_{k+1}$, it undergoes an instantaneous transformation governed by $R_{i_k}^{i_{k+1}}$ (this $T_t(\xi)$ as the analogue for continuous-time OQW of the operator $L_\pi$ of \cite{Pa1,Pa3}). Therefore, the expression for $\rho_t(\xi)$ in \Cref{eq:defXttrhott} encodes the effect of the reduction postulate, or postulate of the collapse of the wave function, on the state of a particle initially at $i_0$ and with internal state $\rho_0$.

This rigorous connection of the unravelling \eqref{eq:prop_dyson_CTOQW2} to (indirect) measurement was first described in \cite{Bar1} (see also \cite{Pelleg1,Pelleg2}, as well as \cite{DdRM} for a connection to two-time measurement statistics).


\smallskip
To summarize this section and the preceding one, we have defined a Markov process $(\mu_t)_t$ as $\mu_t=\rho_t\otimes \ketbra{X_t}{X_t}$, where $X_t\in V$ and $\rho_t\in \mathcal S_{\mathfrak h_{X_t}}$, of which the law can be computed in two ways: either by the Dyson expansion of the CTOQW as in \Cref{eq_prob_dyson} or by use of the stochastic differential equation \eqref{qtcc}. 

\section{Irreducibility of quantum Markov semigroups}
In this section, we state the equivalence between different notions of irreducibility for general quantum Markov semigroup. Our main motivation is the fact that we could not find a complete proof in the case of an infinite-dimensional Hilbert space, as is required e.g.\ for CTOQW with infinite $V$. We then discuss irreducibility for CTOQW.

\begin{thm}\label{irr}
Let $\mathcal{T}:=(\Tcal_t )_{t\geq 0}$ be a quantum Markov semigroup with Lindbladian \begin{equation}\label{lindblad}
\Lcal(\mu)=G\mu+\mu G^*+\displaystyle{\sum_{i\in I}} {L_i}\,\mu\,L_i^* \,.
\end{equation} The following assertions are equivalent:
\begin{enumerate}
\item $\mathcal{T}$ is \emph{positivity improving}: for all $A\in\mathcal I_1(\Kb)$ with  $A\geq 0$ and $A\neq 0$, there exists $t>0$ such that $e^{t\Lcal}(A)>0$.
\item For any $\varphi\in\Kb \backslash\{0\}$, the set $\C[\Lcal]\,\varphi$ is dense in $\Kb$ where $\C[\Lcal]$ is the set of polynomials in $e^{tG}$ for $t>0$ and in $L_i$ for $i\in I$.
\item For any $\varphi\in\Kb \backslash\{0\}$, the set $\C[G,L]\,\varphi$ is dense in $\Kb$ where $\C[G,L]$ is the set of polynomials in $G$ and in the $L_i$ for $i\in I$.
\item $\mathcal{T}$ is irreducible, i.e.\ there exists $t>0$ such that $\Tcal_t$ admits no non-trivial projection $P\in\Bb(\Kb)$ with $\Tcal_t\big(P\mathcal I_1(\Kb) P\big)\subset P\mathcal I_1(\Kb) P $.
  \end{enumerate}
\end{thm}
From now on, any quantum Markov semigroup which satisfies any one of the equivalent statements of Theorem \ref{irr} is simply called \emph{irreducible}.

\begin{rmq}
Positivity improving maps are also called primitive. We therefore call \emph{primitivity} the property of being positivity improving. Remark also that one can replace ``there exists $t>0$'' by ``for all $t>0$'' in assertions $1.$ and $4.$ above to get another equivalent formulation of irreducibility and primitivity. This follows from the observation that assertion $3.$ does not depend of $t$.
\end{rmq}

\begin{proof}
We first prove the equivalence $1.\Leftrightarrow 2.$ Note that $1.$\ holds if and only if for every $\varphi_0\neq 0$, there exists $t_0>0$ such that $\langle \varphi,e^{t\Lcal}(\ketbra{\varphi_0}{\varphi_0})\varphi\rangle >0$ for all $\varphi\neq 0$. Now remark that from Equation \eqref{eq:prop_dyson_CTOQW1},
\begin{align} \label{eq_preuveirreductibilite12}
\langle \varphi,e^{t\Lcal}(\ketbra{\varphi_0}{\varphi_0})\varphi\rangle  
=\sum_{n=0}^\infty\sum_{i_0,\ldots,i_n\in I}\int_{0<t_1<\cdots<t_n<t} |\langle\varphi,\zeta_t(\upxi)\varphi_0\rangle|^2\,\der t_{1}\cdots \der t_{n}
\end{align}
where $\upxi=(i_1,\ldots,i_n;t_1,\ldots,t_n)$. Assume $1.$\ and fix $\varphi_0\neq 0$. If for some $t\geq0$, the left-hand-side of \eqref{eq_preuveirreductibilite12} is positive for any $\varphi\neq 0$, then for any such $\varphi\neq 0$ there exists $\upxi$ with $\braket \varphi{\zeta_t(\upxi)\varphi_0}\neq 0$. Since $\zeta_t(\upxi)\varphi_0\in\C[\mathcal L]\varphi_0$ and the latter is a vector space, this implies that $\C[\mathcal L]\varphi_0$ is dense in $\mathcal K$. Now assume $2.$ and fix $\varphi_0\neq 0$. Since $\C[\mathcal L]\varphi_0$ is dense in $\mathcal K$, for any $\varphi\neq0$ there exists an element $\psi=e^{s_n\,G}\,L_{i_n}\,\cdots\,L_{i_1}\,e^{s_1\,G} \varphi_0$ such that $\braket \varphi\psi\neq 0$. However, for $t\geq s_1+\ldots+s_n$, $\psi$ is of the form $\zeta_t(\upxi)\varphi_0$ for some $\upxi=(i_1,\ldots,i_n;t_1,\ldots,t_n)$. By continuity of $\zeta$ in $t_1,\ldots,t_n$, the right-hand side of \eqref{eq_preuveirreductibilite12} is positive and this proves $1$.

To prove the equivalence $2.\Leftrightarrow 3. $, we use the fact that $G=\lim_{t \to 0}(\mathrm e^{tG}-\id)/t$, which implies that for any $\varphi\in\Kcal \backslash\{0\}$, 
\[\C[G,L]\varphi\subset\overline{\C[\Lcal]}\varphi\subset\overline{\C[\Lcal]\varphi} \, .\]
Since $e^{tG} =\lim_{n \to \infty}\sum_{k=0}^{n}{t^k G^k}/{k!}$, for any $\varphi\in\Kcal \backslash\{0\}$ we also have
\[\C[\Lcal]\varphi\subset\overline{\C[G,L]}\varphi\subset\overline{\C[G,L]\varphi} \, .\]
Therefore, for any $\varphi\in\Kcal \backslash\{0\}$,  
\begin{equation}\label{22}
\C[\Lcal]\varphi \text{ is dense in } \Kcal\Longleftrightarrow  \C[G,L]\varphi \text{ is dense in } \Kcal \, .
\end{equation}

The implication $1.\Rightarrow 4. $ is obvious. It remains to prove that $4.\Rightarrow 2. $. To this end, suppose that $\mathcal{T}$ is irreducible. Let $\varphi\in \Kb\backslash\{0\}$ and denote by $P$ the orthogonal projection on $\overline{\C[\Lcal]\varphi}$. The goal is to prove that $P=\id$. For all $\psi\in\Kb\backslash\{0\}$, 
\begin{align*}
e^{t\Lcal}(P\vert \psi\rangle\langle \psi\vert P)&=\sum_{n=0}^\infty\sum_{i_0,\ldots,i_n\in I}\int_{0<t_1<\cdots<t_n<t}\zeta_t(\upxi)P \vert \psi\rangle\langle \psi\vert P\zeta_t(\upxi)^*\der t_{1}\cdots \der t_{n}\\
&=\sum_{n=0}^\infty\sum_{i_0,\ldots,i_n\in I}\int_{0<t_1<\cdots<t_n<t} \vert \zeta_t(\upxi)P \psi\rangle\langle \zeta_t(\upxi)P\psi\vert \der t_{1}\cdots \der t_{n}\, , \end{align*}
Since $\zeta_t(\upxi)\in\C[\Lcal]$ and $P \psi\in\overline{\C[\Lcal]\varphi}$, we have $\zeta_t(\upxi)P \psi\in\overline{\C[\Lcal]\varphi}$ and thus
\[\Tcal_t(P\vert \psi\rangle\langle \psi\vert P)=P\,\Tcal_t(P\vert \psi\rangle\langle \psi\vert P)P \, .\]
The projection P being subharmonic for $\Tcal_t$ which is irreducible by assumption, $P$ is trivial. As it is non-zero, $P=\id$. Since $P$ is the orthogonal projection on $\overline{\C[\Lcal]\varphi}$, this shows that $\C[\Lcal]\varphi$ is dense in $\Kb$. \end{proof}
\begin{rmq} \label{rmq_irrLLstar}
An immediate corollary of \Cref{irr} is that a quantum Markov semigroup $\Tcal=(\Tcal_t)_t$ is irreducible if and only if its adjoint $\Tcal^*=(\Tcal^*_t)_t$ is irreducible.
\end{rmq}

We now introduce the notion of irreducibility of a CTOQW, focusing on the trajectorial formulation. Let $\mathcal{T}:=(\Tcal_t )_{t\geq 0}$ be a CTOQW on a set $V$. For $i,j$ in $V$ and $n\in\N$, we denote by $\Pcal^n(i,j)$  the set of continuous-time trajectories going from $i$ to $j$ in $n$ jumps:
\[\mathcal{P}^n(i,j)=\{\xi=(i_0,\ldots,i_n;t_1,\ldots,t_n)\in \Xi^{(n)}_{\infty}\,|\, i_0=i, \, i_n=j\} \, \] 
and we set $\mathcal{P}(i,j)=\cup_{n\in\N}\mathcal{P}^n(i,j)$. For any $\xi=(i,\ldots,j;t_1,\ldots,t_n)$ in $\mathcal{P}(i,j)$, we denote by $R(\xi)$ the operator from $\h_i$ to $\h_j$ defined by
\begin{equation}\label{eq_traj}
R(\xi)=R_{i_{n-1}}^{j}e^{(t_{n}-t_{n-1}) G_{i_{n-1}}}R_{i_{n-2}}^{i_{n-1}}  \,  \cdots \,e^{(t_{2}-t_{1}) G_{i_{1}}}R_{i}^{i_1} e^{t_{1} G_{i}}\, .
\end{equation}
This is almost the same as the operator $T_t(\xi)$ defined in Equation \eqref{eq:prop_dyson_CTOQW2} but here we do not include the evolution between the time $t_n$ and $t$, that is,
\[T_t(\xi)=e^{(t-t_n)\,G_j}\,R(\xi)\,.\]
The following proposition is a direct application of \Cref{irr}, and will constitute our definition of irreducibility for continuous-time open quantum walks. The criterion here is equivalent to any other formulation proposed in Theorem \ref{irr}.

\begin{prop}\label{posimppath}
The CTOQW defined by the quantum Markov semigroup $\mathcal{T}:=(\Tcal_t )_{t\geq 0}$ is \emph{irreducible} if and only if, for every $i$
and $j$ in $V$ and for any $\varphi$ in $ \h_i\backslash\{0\}$, the set $\{R(\xi)\, \varphi,\, \xi\in\mathcal{P}(i,j)\} $ is total in $ \h_j$.
\end{prop}
The proposition below gives a sufficient condition for irreducibility of a quantum Markov semigroup. We recall (see \cite{schrader}) that a map on $\mathcal I_1(\mathcal K)$ of the form $\mu\mapsto \sum_i L_i \mu L_i^*$ is called irreducible if for any $\varphi_0\neq 0$, the set $\C[L]\varphi_0$ is dense in $\mathcal K$ (this is simply a discrete-time analogue of the present notion of irreducibility).
\begin{prop}
Let $\mathcal{T}:=(\Tcal_t )_{t\geq 0}$ be a quantum Markov semigroup with Lindbladian $$\Lin(\mu)=G\mu+\mu G^*+\displaystyle{\sum_{i}} {L_i}\,\mu\,L_i^* \, .$$ If $\Phi(\mu)=\sum_{i} {L_i}\,\mu\,L_i^*$ is irreducible, then  $\mathcal{T}:=(\Tcal_t)_{t\geq 0}$ is irreducible as well.
\end{prop}
\begin{proof}
This is obvious from Lemma 3.3 of \cite{schrader} and the third characterization of irreducibility in Theorem \ref{irr}.
\end{proof}

One can compare the present notion of irreducibility with the usual one for classical Markov chain. Actually if $Q=(q_{i,j})_{i,j\in V}$ denotes the generator of a continuous-time Markov chain, then the irreducibility of the chain depends on the transitions $(q_{i,j},i\neq j)$. In particular the Markov chain is irreducible if and only if for all $i,j$ there exists a path $i_0=i,i_1,\ldots,i_n=j$ such that $q_{i_0,i_1}\times\ldots\times q_{i_{n-1},i_n}>0$. Unfortunately in the case of CTOQW only one implication is valid since one can find $\mathcal{T}:=(\Tcal_t)_{t\geq 0}$ irreducible where the corresponding $\Phi$ is not irreducible (see the example below). 
\begin{exe} We focus on an example on $\mathcal \Hb= \C^2\otimes \vert 1\rangle+\C^2\otimes \vert 2\rangle$, where the Lindbladian is defined by \eqref{eq_lind_alternative} with: 
\[G_1=G_2=
	\frac{1}{2}\begin{pmatrix} -1&2\\-2&-1 \end{pmatrix},
\ R_1^2=R_2^1=
	\begin{pmatrix}	0&1\\1&0\end{pmatrix}
\, .\]
One can easily check that $\{R(\xi)\, \varphi,\, \xi\in\mathcal{P}(i,j)\}=\h_j $ for all $i,j\in\{1,2\}$ and $\varphi\in\h_i\setminus\{0\}$ but the discrete OQW defined by $R_1^2$ and $R_2^1$ is not irreducible. This is an example of CTOQW where $e^{t\Lin}$ is irreducible even though $\Phi$ is not. 
					 
					 
\end{exe}

\section{Transience and recurrence of irreducible CTOQW}

In the classical theory of Markov chains on a finite or countable graph, an irreducible Markov chain can be either transient or recurrent. Transience and recurrence issues are central to the study of Markov chains and help describe the Markov chain's overall structure. In the case of CTOQW, transience and recurrence notions are made more complicated by the fact that the process $(X_t)_t$ alone is not a Markov chain.

In the present section, we define the notion of recurrence and transience of a vertex in our setup and prove a dichotomy similar to the classical case, based on the average occupation time at a vertex. However, compared to the classical case, the relationship between the occupation time and the first passage time at the vertex is less straightforward. Recall that the first passage time at a given vertex $i\in V$ is defined as
\[\tau_i=\inf\{t\geq T_1|X_t=i\}\,\]
where $T_1$ is defined in \Cref{sectQtraj}. Similarly the occupation time is given by  
\[n_i=\int_{0}^{\infty}\ind_{X_t=i} \ \der t \,.\]
In the discrete-time and irreducible case (Theorem 3.1.\ of \cite{Pa1}), the authors prove that there exists a trichotomy rather than the classical dichotomy. We state a similar result for continuous-time semifinite open quantum walks (we recall that an OQW is semifinite if $\dim\h_i<\infty$ for all $i\in V$).

\begin{thm}\label{trich}
Consider a semifinite irreducible continuous-time open quantum walk. Then we are in one (and only one) of the following situations:
\begin{enumerate}
\item For any $i,j$ in $V$ and $\rho$ in $\Sb_{\h_i}$, one has $\E_{i,\rho}(n_j)=\infty $ and $\Pro_{i,\rho}(\tau_j<\infty)=1$.
\item For any $i,j$ in $V$ and $\rho$ in $\Sb_{\h_i}$, one has $\E_{i,\rho}(n_j)<\infty $ and $\Pro_{i,\rho}(\tau_i<\infty)<1$.
\item For any $i,j$ in $V$ and $\rho$ in $\Sb_{\h_i}$, one has $\E_{i,\rho}(n_j)<\infty $, but there exist $i$ in $V$ and $\rho,\rho'$ in $\Sb_{\h_i}$ ($\rho$ necessarily non-faithful) such that $\Pro_{i,\rho}(\tau_i<\infty)=1$ and $\Pro_{i,\rho'}(\tau_i<\infty)<1$.
\end{enumerate}
\end{thm}

Note that in the sequel we only focus on semifinite case. Recall that when $\h_i$ is one-dimensional for all $i\in V$, we recover classical continuous-time Markov chains. In this case, the Markov chain falls in one of the first two categories of this theorem; that is, the third category is a specifically quantum situation.

The rest of this section is dedicated to the proof of Theorem \ref{trich}. More precisely, in Subsection \ref{subsect_def_trans} we prove the dichotomy between infinite and finite average occupation time. This allows us to define transience and recurrence of CTOQW. We also give examples of CTOQW that fall in each of the three classes of Theorem \ref{trich}. In Section \ref{subsect_tech} we state technical results that give closed expressions to the occupation time and the first passage time. Finally, the proof of Theorem \ref{trich} is given in Subsection \ref{subsect_proof}.

\subsection{Definition of recurrence and transience}\label{subsect_def_trans}

We begin by proving that for an irreducible CTOQW, the average occupation time $\E_{i,\rho}(n_j)$ of site $j$ starting from site $i$ is either finite for all $i,j$ or infinite for all $i,j$.

\begin{prop}\label{12dist}
Consider a semifinite irreducible continuous-time open quantum walk. Suppose furthermore that there exist $i_0,j_0\in V$ and $\rho_0\in \Sb_{\h_{i_0}} $ such that $\E_{i_0,\rho_0}(n_{j_0})=\infty$. Then, for all $i,j\in V$ and $\rho\in \Sb_{\h_i} $ one has $\E_{i,\rho}(n_{j})=\infty$. 
\end{prop}
\begin{proof}
Fix $i,j \in V$ and $\rho\in \Sb_{\h_i} $. Then one has
\[\E_{i,\rho}(n_{j})=\int_{0}^{\infty}\Pro_{i,\rho}(X_t=j) \,\der t=\int_{0}^{\infty}\Tr\big(e^{t\Lcal}(\rho\otimes|i\rangle\langle i|)(I\otimes |j\rangle\langle j|)\big) \,\der t\,.\]
By hypothesis, $(\Tcal_t)_{t\geq 0}$ is irreducible and thus positivity improving by Theorem \ref{irr}; by \Cref{rmq_irrLLstar} the same is true of $(\Tcal^*_t)_{t\geq 0}$. Therefore, since for any $i\in V$, $\h_i$ is finite-dimensional, for any $s>0$ there exist scalars $\alpha,\beta>0$ such that
\[e^{s\Lcal}(\rho\otimes|i\rangle\langle i|)\geq \alpha \ \rho_0\otimes|i_0\rangle\langle i_0|\ \mbox{ and } \ e^{s\Lcal^*}(I\otimes|j\rangle\langle j|)\geq \beta \ I\otimes|j_0\rangle\langle j_0|   \, .\]
We then have, fixing $s>0$,
\begin{align*}\E_{i,\rho}(n_{j})
&\geq\int_{2s}^{\infty}\Tr\big(e^{(t-2s)\Lcal}\big(e^{s\Lcal}(\rho\otimes|i\rangle\langle i|)\big)\, e^{s\Lcal^*}(I\otimes |j\rangle\langle j|)\big) \ \der t\\
&\geq \alpha\beta \int_{0}^{\infty}\Tr\big(e^{u\Lcal}(\rho_0\otimes|i_0\rangle\langle i_0|)(I\otimes |j_0\rangle\langle j_0|)\big) \ \der u\\
&\geq \alpha\beta\  \E_{i_0,\rho_0}(n_{j_0}) \ .\end{align*}
This concludes the proof.
\end{proof}

The above proposition leads to a natural definition of recurrent and transient vertices of $V$.
\begin{defi}
For any continuous-time open quantum walk, we say that a vertex $i$ in $V$ is:
\begin{itemize}
\item recurrent if for any $\rho\in\Scal_{\h_i}$, $\E_{i,\rho}(n_{i})=\infty$;
\item transient if there exists $\rho\in\Scal_{\h_i}$ such that $\E_{i,\rho}(n_{i})<\infty$.
\end{itemize}
Thus, by Proposition \ref{12dist}, for an irreducible CTOQW, either all vertices are recurrent or all vertices are transient. Furthermore, in the transient case, $\E_{i,\rho}(n_{i})<\infty$ for all $\rho$ in $\Scal_{\h_i}$.
\end{defi}

We conclude this section by illustrating Theorem \ref{trich} by simple examples. The $n$-th example corresponds to the $n$-th situation in Theorem \ref{trich}.

\begin{exe}\
\begin{enumerate}
\item For $V=\{0,1\}$ and $\h_0=\h_1=\C$, consider the CTOQW characterized by the following operators:
\[G_0=G_1=-\frac{1}{2}\,,\qquad R_0^{1}=R_1^0=1 \, .\]
Then the process $(X_t)_{t\geq0}$ is a classical continuous Markov chain on $\{0,1\}$, where the walker jumps from one site to the other after an exponential time of parameter 1.  
\item For $V=\Z$ and $\h_i=\C$ for all $i\in\Z$, consider the CTOQW described by the transition operators:
\[G_i=-\frac{1}{2}\,,\qquad R_i^{i+1}=\frac{\sqrt 3}{ 2}\,,\qquad R_i^{i-1}=\frac{1}{2}, \text{ for all } i\in\Z\, .\]
The process $(X_t)_{t\geq0}$ is a classical continuous Markov chain on $\Z$ where after an exponential time of parameter 1, the walker jumps to the right with probability $\frac{3}{4}$ or to the left with probability $\frac{1}{4}$.
\item Consider the CTOQW defined by $V = \N$ with $\h_1=\C^2$ and $\h_0=\h_i=\C$ for $i\geq 2$, and 
\[G_0=-\frac{1}{2}\,,\quad R_0^1= \frac{1}{\sqrt 5}\begin{pmatrix}
2\\
1
\end{pmatrix}\,,\quad G_1= -\frac{1}{2}I_{2}\,, R_1^0= \begin{pmatrix}
0&1
\end{pmatrix}\,,\quad  R_1^2= \begin{pmatrix}
1&0
\end{pmatrix}\,,\]
\[ R_2^1= \frac{1}{2\sqrt 2}\begin{pmatrix}
1\\
1
\end{pmatrix}\,,\quad G_i=-\frac{1}{2}\,,\quad R_i^{i+1}=\frac{\sqrt 3}{ 2} \text{ for } i\geq 2 \text{ and } R_i^{i-1}=\frac{1}{2}\text{ for } i\geq 3\, .\]
This is an example of positivity improving CTOQW where, for $\rho=\begin{pmatrix}
0&0\\
0&1
\end{pmatrix}$, one has  $$\Pro_{1,\rho}(\tau_1<\infty)=1$$ but $$\Pro_{i,\rho'}(\tau_i<\infty)<1$$ for any $\rho'\neq\begin{pmatrix}
0&0\\
0&1
\end{pmatrix}$. This example therefore exhibits ``specifically quantum'' behavior. This example is inspired from \cite{Pa1}.
\end{enumerate}
\end{exe}

\subsection{Technical results}\label{subsect_tech}

\Cref{Pij} below is essential, as it expresses the probability of reaching a site in finite time as the trace of the initial state, evolved by a certain operator.

\begin{prop}\label{Pij}
For any continuous-time open quantum walk, there exists a completely positive linear operator $\mathfrak{P}_{i,j}$ from $\mathcal{I}(\h_i)$ to $\mathcal{I}(\h_j)$ such that for every $i,j\in V$ and $\rho\in \Sb_{\h_i} $,
\[\Pro_{i,\rho}(\tau_j<\infty)=\Tr\big(\mathfrak{P}_{i,j}(\rho)\big) \, .\]
Furthermore, the map $\mathfrak{P}_{i,j}$ can be expressed by:
\[\mathfrak{P}_{i,j}(\rho)=\sum_{n=0}^\infty\displaystyle\sum_{\substack{{i_1,\ldots,i_{n-1} \in V\backslash\{j\}} \\ i_0=i, i_n=j}}\int_{0<t_1<\ldots<t_n<\infty} R(\xi)\, \rho\,R(\xi)^*\der t_{1}\ldots \der t_{n}\,,\] 
where we recall that $
R(\xi)=R_{i_{n-1}}^{i_n}\,e^{(t_{n}-t_{n-1}) G_{i_{n-1}}}R_{i_{n-2}}^{i_{n-1}}  \,  \ldots \,R_{i_0}^{i_1}\, e^{t_{1} G_{i_0}}$ for $\xi=(i_0,\ldots,i_n;t_1,\dots,t_n)$.
\end{prop}
Note that we do not require the $\h_i$ to be finite-dimensional here.
\begin{proof}We have the trivial identity:
\begin{align}\label{eq:tauj}\Pro_{i,\rho}(\tau_j<t)&=\sum_{n=0}^\infty~\displaystyle\sum_{\substack{{i_1,\ldots,i_{n-1} \in V\backslash\{j\}} \\ i_0=i, i_n=j}}~\int_{0<t_1<\ldots<t_n<t} \Tr\big(e^{(t-t_n)\Lcal}\big(R(\xi)\, \rho\,R(\xi)^* \otimes |j\rangle\langle j|\big)\big)\der t_{1}\ldots \der t_{n} \, .\end{align}
Then, since $e^{(t-t_n)\Lcal}$ is trace preserving, 
\[\Pro_{i,\rho}(\tau_j<t)
=\sum_{n=0}^\infty~\displaystyle\sum_{\substack{{i_1,\ldots,i_{n-1} \in V\backslash\{j\}} \\ i_0=i, i_n=j}}~\int_{0<t_1<\ldots<t_n<t} \Tr\big(R(\xi)\, \rho\,R(\xi)^* \big)\der t_{1}\ldots \der t_{n} \, ,\]
and since both sides of the identity are nondecreasing in $t$, taking the limit $t\to+\infty$ yields
\[\Pro_{i,\rho}(\tau_j<\infty)
=\sum_{n=0}^\infty~\displaystyle\sum_{\substack{{i_1,\ldots,i_{n-1} \in V\backslash\{j\}} \\ i_0=i, i_n=j}}~\int_{0<t_1<\ldots<t_n<\infty} \Tr\big(R(\xi)\, \rho\,R(\xi)^* \big)\der t_{1}\ldots \der t_{n} \, .\]
It remains to show that $\mathfrak{P}_{i,j}$ is well defined. Let us denote by $(V_n)_{n\in \N}$ an increasing sequence of subsets of $V$ such that $|V_n|=\min (n,|V|)$ and $\bigcup_{n\in \N} V_n=V$. For any $X\in \mathcal{I}(\h_i)\backslash \{0\}$ write the canonical decomposition $X=X_1-X_2+\mathrm i X_3- \mathrm i X_4$ of $X$ as a linear combination of four nonnegative operators. We get
\begin{align*}
&\Tr\,\Big|\sum_{n=0}^N ~\displaystyle\sum_{\substack{{i_1,\ldots,i_{n-1} \in V_N\backslash\{j\}} \\ i_0=i, i_n=j}}~\int_{0<t_1<\ldots<t_n<N} R(\xi)\, X\,R(\xi)^* \,\der t_{1}\ldots \der t_{n}\Big|\\
&\leq \sum_{m=1}^4 \Tr\,\Big|\sum_{n=0}^N ~\displaystyle\sum_{\substack{{i_1,\ldots,i_{n-1} \in V_N\backslash\{j\}} \\ i_0=i, i_n=j}}~\int_{0<t_1<\ldots<t_n<N} R(\xi)\, X_m\,R(\xi)^* \,\der t_{1}\ldots \der t_{n}\Big|\\
&\leq\sum_{m=1}^4\Tr\,X_m \times \sum_{n=0}^N~\displaystyle\sum_{\substack{{i_1,\ldots,i_{n-1} \in V_N\backslash\{j\}} \\ i_0=i, i_n=j}}~\int_{0<t_1<\ldots<t_n<N} \Tr\big(R(\xi)\, \frac{X_m}{\Tr(X_m) }\,R(\xi)^* \big)\,\der t_{1}\ldots \der t_{n}\\
&\leq\sum_{m=1}^4 \Tr\,X_m \times\Pro_{i,\frac{X_m}{\Tr(X_m)}}(\tau_j<N)\\
&\leq \sum_{m=1}^4\Tr\,X_m\\
&\leq 2\Tr\,|X|
\end{align*}
(alternatively apply Theorem 5.17 in \cite{wolftour} to $X_1-X_2$ and $X_3-X_4$). Then
\[\sup_N \,\Tr\,\Big|\sum_{n=0}^N ~\displaystyle\sum_{\substack{{i_1,\ldots,i_{n-1} \in V_N\backslash\{j\}} \\ i_0=i, i_n=j}}~\int_{0<t_1<\ldots<t_n<N} \,R(\xi)\, X\,R(\xi)^*\, \der t_{1}\ldots \der t_{n}\Big|<\infty \, . \]
Consequently, by the Banach-Steinhaus Theorem, the operator on $\mathcal{I}(\h_i)$ to $\mathcal{I}(\h_j)$ defined by
\[\mathfrak{P}_{i,j}(X)=\sum_{n=0}^\infty~\displaystyle\sum_{\substack{{i_1,\ldots,i_{n-1} \in V\backslash\{j\}} \\ i_0=i, i_n=j}}~\int_{0<t_1<\cdots<t_n<\infty} R(\xi)\, X\,R(\xi)^*\,\der t_{1}\cdots \der t_{n} \]
is everywhere defined and bounded.
\end{proof}

As a corollary, using the definition of the operator $\mathfrak{P}_{i,j}$ for $i,j\in V$, we obtain a useful expression for $\E_{i,\rho}(n_j)$:

\begin{cor}\label{core}
For every $i,j\in V$ and $\rho\in \Sb_{\h_i} $, we have
\begin{equation}\label{eq_core}
\E_{i,\rho}(n_j)=\sum_{k=0}^\infty\Tr\big(\mathfrak{P}_{j,j}^k\circ\mathfrak{P}_{i,j}(\rho)\big)\,.
\end{equation}
\end{cor}
\begin{proof}
Let $i,j\in V$ and $\rho\in \Sb_{\h_i} $. Then
\begin{align*}&\E_{i,\rho}(n_{j})=\int_{0}^{\infty}\Pro_{i,\rho}(X_t=j) \ \der t=\int_{0}^{\infty}\Tr\big(e^{t\Lcal}(\rho\otimes|i\rangle\langle i|)(\id\otimes |j\rangle\langle j|)\big)\,\der t\\
&=\Tr\Big(\sum_{n=0}^\infty\displaystyle\sum_{k=0}^n\sum_{\substack{{m_1,\ldots,m_{k}\in\N}\\ m_1<\cdots<m_k=n}}\sum\limits_{\substack{{i_1,\ldots,i_{{m_1}-1},i_{{m_1}+1},\ldots,i_{{m_k}-1} \in V\backslash\{j\}} \\ i_0=i,\, i_{m_1}=i_{m_2}=\cdots=i_{m_k}=j}}\int_{0<t_1<\cdots<t_n<t} \Upsilon\rho \Upsilon^* \,\der t_{1}\cdots \der t_{n}\,\der t\Big) \, ,
\end{align*}
where $\Upsilon=R_{i_{n-1}}^{i_n}e^{(t_{n}-t_{n-1}) G_{i_{n-1}}}R_{i_{n-2}}^{i_{n-1}}  \,\cdots   R_{i_{m_1}}^{i_{{m_1}+1}}e^{(t_{{m_1}+1}-t_{{m_1}}) G_{i_{m_1}}}R_{i_{{m_1}-1}}^{i_{m_1}}\cdots \,R_{i_0}^{i_1} e^{t_{1} G_{i_0}} $.
The above expression corresponds to a decomposition of any path from $i$ to $j$ as a concatenation of a path from $i$ to $j$ and $k$ paths from $j$ to $j$ which do not go through $j$. This yields \Cref{eq_core}.
\end{proof}

The next corollary allows us to link the quantity $\Pro_{i,\rho}(\tau_j<\infty)$ to the adjoint of the operator $\mathfrak{P}_{i,j}$. In particular, as we shall see, it is a first step towards linking the properties of $\Pro_{i,\rho}(\tau_j<\infty)$ and $\E_{i,\rho}(n_{j})$.

\begin{cor}\label{cordij}
Let $i$ and $j$ be in $V$. One has
 \[\Pro_{i,\rho}(\tau_j<\infty)=1\iff\mathfrak{P}_{i,j}^*(\id)=\begin{pmatrix}
	\id&0\\
	0&*\end{pmatrix} \text{ in the decomposition } \h_i=\ran \rho\oplus  (\ran \rho)^\perp \, .\]  In particular, if there exists a faithful $\rho$ in $\Sb_{\h_i}$ such that $\Pro_{i,\rho}(\tau_j<\infty)=1$, then $\Pro_{i,\rho'}(\tau_j<\infty)=1$ for any $\rho'$ in $\Sb_{\h_i} $.
\end{cor}

\begin{proof}
By \Cref{Pij}, one has $\Pro_{i,\rho}(\tau_j<\infty)=\Tr(\rho \, \mathfrak{P}_{i,j}^*(\id))$. Therefore, if $\Pro_{i,\rho}(\tau_j<\infty)=1$, then  $\mathfrak{P}_{i,j}^*(\id)$ has the following form in the decomposition $\h_i=\ran \rho \oplus (\ran \rho)^\perp$:
\[\mathfrak{P}_{i,j}^*(\id)=\begin{pmatrix}
		\id&A\\
		A&B
		\end{pmatrix} \,  .\]
Besides, the fact that $\id\geq\mathfrak{P}_{i,j}^*(\id)$ forces $A$ to be null. In particular, if $\rho$ is faithful, then $\mathfrak{P}_{i,j}^*(\id)=\id$ and therefore $\Pro_{i,\rho'}(\tau_j<\infty)=1$ for any $\rho'$ in $\Sb_{\h_i} $. 
\end{proof}
	
\subsection{Proof of Theorem \ref{trich}}\label{subsect_proof}

Let $i$ and $j$ be in $V$. As we can see in \Cref{cordij}, if we suppose that $\Pro_{i,\rho}(\tau_j<\infty)=1$ for a faithful density matrix $\rho$, we necessary have $\mathfrak{P}_{i,j}^*(\id)=\id$. This will be used in the following proposition, which in turn explains the statement regarding non-faithfulness in the third category of Theorem \ref{trich}. 

\begin{prop}\label{yet}
	Let $i$ be in $V$. If there exists a faithful $\rho$ in $\Sb_{\h_i}$ such that $\Pro_{i,\rho}(\tau_i<\infty)=1$, then one has $\E_{i,\rho'}(n_i)=\infty$ for any $\rho'$ in $\Sb_{\h_i} $.
\end{prop}
	
\begin{proof}
We set $\tau_1^i=\tau_i$ and, for all $n>1$, we define $\tau_i^{(n)}$ as the time at which $(X_t)_{t\geq0}$ reaches $i$ for the $n$-th time:
\[\tau_i^{(n)}=\inf\{t>\tau_{n-1}^{i}|\,X_t=i \text{ and }X_{t-}\neq i\}\,.\]
From Corollary \ref{cordij}, one has $\Pro_{i,\rho'}(\tau_i<\infty)=1$ for all $\rho'$ in $\Sb_{\h_i}$. This implies that for all $n>0$, $\tau_i^{(n)}$ is $\Pro_{i,\rho'}$-almost finite for any $\rho'\in \Sb_{\h_i}$. For $n\geq 0$, let $T_n^i$ be the occupation time in $i$ between $\tau_n^i$ and $\tau_{n+1}^i$:
\[T_n^i=\inf\{u\,|\,X_{\tau_i^{(n)}+u}\neq i\}\]
with the convention that $\tau_i^{(0)}=0$. Since we have
\begin{equation*}
\E_{i,\rho'}(n_i)\geq\E_{i,\rho'}\big(\sum_{n\geq1}T_{n}^i\big)\geq \sum_{n\geq 1}\inf_{\hat\rho\in \Sb_{\h_i}}\E_{i,\hat\rho}(T_{n}^i),
\end{equation*}
it will be enough to obtain a lower bound for $\E_{i,\hat\rho}(T_n^i)$ which is uniform in $n$ and in $\hat\rho$. To this end, we use the quantum trajectories defined in \eqref{qtcc}. We first compute $\Pro_{i,\hat\rho}(T_n^i>t)$ for all $t\geq 0$. To treat the case of $n=1$ we consider the solution of
\begin{equation}\label{EDO1}\eta_t^{\hat\rho}=\hat\rho+\int_0^t \big(G_i\,\eta_s^{\hat\rho}+\eta_s^{\hat\rho}\,G_i^*-\eta_s^{\hat\rho}\,\Tr(G_i\,\eta_s^{\hat\rho}+\eta_s^{\hat\rho}\,G_i^*)\big)\,ds\,. \end{equation} Using the independence of the Poisson processes $N^{i,j}$ involved in \eqref{qtcc} we get
\begin{align}
\Pro_{i,\hat\rho}(T_1^i>t)&=\Pro_{i,\hat\rho}(\textrm{no jump has occurred before time t})\nonumber\\
&=\pp_{i,\hat\rho}\big(N^{i,j}\big(\{u,y\,\vert\, 0\leq u\leq t,\, 0\leq y\leq\Tr(R_{i}^j\eta_u^{\hat\rho}{R_{i}^j}^*)\}\big)=0\ \forall j\neq i\big)\nonumber\\
&=\prod_{j\neq i}\pp_{i,\hat\rho}\big(N^{i,j}\big(\{u,y\,\vert\, 0\leq u\leq t,\, 0\leq y\leq\Tr(R_{i}^j\eta_u^{\hat\rho}{R_{i}^j}^*)\}\big)=0\big)\nonumber\\
&=\prod_{j\neq i}\,\exp\big\{-\int_0^t\Tr(R_i^j\eta_s^{\hat\rho}R_i^{j*})\,\der s\big\}
\nonumber\\
&=\exp\Big(\int_0^t\Tr\big((G_i+G_i^*)\,\eta_s^{\hat\rho}\big)\,\der s\Big)\,
\end{align}
where we used relation \eqref{eq_Gi}. Similarly, using the strong Markov property,
\begin{align*}
\Pro_{i,\hat\rho}(T_n^i>t)&=\E_{i,\hat\rho}(\mathbf 1_{T_n^i>t})\nonumber\\
&=\E_{i,\hat\rho}\big(\E_{i,\rho_{\tau_n^i}}(\mathbf 1_{T_1^i>t})\big)\nonumber\\
&=\E_{i,\hat\rho}\Big(\exp\big(\int_0^t \Tr\big((G_i+G_i^*)\, \eta_s^{\rho_{\tau_n^i}}\big)\,\der s\big)\Big)\\
&\geq e^{-t \|G_i+G_i^*\|_\infty} \, . 
\end{align*}
Now, using the fact that $\E_{i,\hat\rho}(T_n^i)=\int_0^\infty\Pro_{i,\hat\rho}(T_n^i>t)\,\der t$, this gives us the expected lower bound:
\[\E_{i,\hat\rho}(T_n^i)\geq \frac{1}{\|G_i+G_i^*\|_\infty} \, .\]
This concludes the proof.
\end{proof}
The next proposition is connected to the first point of \Cref{trich}.
\begin{prop}\label{la312}
Consider a semifinite irreducible continuous-time open quantum walk. If there exist $i,j$ in $V$ and $\rho\in \Sb_{\h_i}$ such that $\E_{i,\rho}(n_j)=\infty$, then one has $\Pro_{j,\rho'}(\tau_j<\infty)=1$ for any $\rho'$ in $\Sb_{\h_j} $.   
\end{prop}
\begin{proof} By Proposition \ref{posimppath}, there is no nontrivial invariant subspace of $\h_j$ left invariant by $R(\xi)$ for all $\xi\in\mathcal P (j,j)$. Since any such $\xi$ is a concatenation of paths from $j$ to $j$ that remain in $V\setminus\{j\}$ except for their start- and endpoints, there is also no nontrivial invariant subspace of $\h_j$ left invariant by the operator $\mathfrak{P}_{j,j} $ of Proposition \ref{Pij}. The latter is therefore a completely positive irreducible map acting on the set of trace-class operators on $\h_j$. By the Russo-Dye Theorem (see \cite{RD}), one has $\|\mathfrak{P}_{j,j}\|=\|\mathfrak{P}_{j,j}^*(\id)\|\leq 1 $, so that the spectral radius $\lambda$ of $\mathfrak{P}_{j,j} $ satisfies $\lambda\leq1$. By the Perron-Frobenius Theorem of Evans and Hoegh-Kr\o hn (see \cite{EHK} or alternatively Theorem 3.1 in \cite{schrader}), there exists a faithful density matrix $\rho'$ on $\h_j$ such that $\mathfrak{P}_{j,j}(\rho')=\lambda\rho' $. If $\lambda<1 $, then by Corollary \ref{core} one has $\E_{j,\rho'}(n_j)<\infty$, but then Proposition \ref{12dist} contradicts our running assumption that $\E_{i,\rho}(n_j)=\infty$. Therefore $\lambda=1$ and $\rho'$ is a faithful density matrix such that $\Pro_{j,\rho'}(\tau_j<\infty)=\Tr\big(\mathfrak{P}_{j,j}(\rho')\big)=\Tr(\rho')=1$. We then conclude by \Cref{cordij}.
\end{proof}



\begin{prop}\label{ijbis}
Consider a semifinite irreducible continuous-time open quantum walk; if there exists $i\in V$ such that for all $\rho'\in \Sb_{\h_i}$ one has $\Pro_{i,\rho'}(\tau_i<\infty)=1$, then $\Pro_{i,\rho}(\tau_j<\infty)=1$ for any $j\in V$ and $\rho\in \Sb_{\h_i}$.
\end{prop}
\begin{proof}
Fix $i$ and $j$ in $V$. Observe first that, by irreducibility, for any $\rho$ in $\Sb_{\h_i}$, there exists 
\[\xi=(i=i_0,i_1,\ldots,i_{n-1},i_n=j;t_1,\ldots,t_n)\]
such that $\Tr \big(R(\xi)\rho R(\xi)^*\big)>0$. We denote by $t(\xi)$ the element $t_n$ of $\xi$.
Using the continuity of $\Tr\big(R(\xi)\rho R(\xi)^*\big)$ in $\rho$ and the compactness of $\Sb_{\h_i}$, we obtain a finite family $\xi_1,\ldots,\xi_p$, of paths, again going from $i$ to $j$, such that
\[\inf_{\rho\in\Sb_{\h_i}} \max_{k=1,\ldots,p} \Tr\big(R(\xi_k)\rho R(\xi_k)^*\big)>0\,. \]
Let $\delta>\max_{k=1,\ldots,p} t(\xi_k)$. By continuity of each $\Tr\big(R(\xi_i)\rho R(\xi_i)^*\big)$ in the underlying jump times $t_1,\ldots,t_n$ and using expression \eqref{eq:tauj}, we have 
\[\alpha:= \inf_{\rho\in\Sb_{\h_i}} \pp_{i,\rho}(\tau_j\leq \delta)>0\,.\]
Now, if $\pp_{i,\rho}(\tau_i<\infty)=1$ for all $\rho$ in $\Sb_{\h_i}$, then the discussion in \Cref{sectQtraj} imply that almost-surely one can find an increasing sequence $(t_n)_n$ of times with $t_n\to\infty$ and $x_{t_n}=i$. Choose a subsequence $(t_{\varphi(n)})_n$ such that $t_{\varphi(n)}-t_{\varphi(n-1)}>\delta$ for all $n$. Since never reaching $j$ means in particular not reaching $j$ between $t_{\varphi(n)}$ and $t_{\varphi(n+1)}$ for $n=1,\ldots,k$, the Markov property of $(X_t,\rho_t)_{t\geq0}$ and the lower bound $t_{\varphi(n)}-t_{\varphi(n-1)}>\delta$ imply that for all $\rho\in\Scal_{\h_i}$,
\begin{align*}
\pp_{i,\rho}(\tau_j=\infty)
& \leq \pp_{i,\rho}\left(\forall 0\leq n\leq k\,,\,\forall t\in[t_{\varphi(n)},t_{\varphi(n+1)}]\,,\,X_t\ne j\right) \\
& \leq \big(\sup_{\rho\in\h_i}\pp_{i,\rho}(\tau_j>\delta)\big)^k  \\
& \leq (1-\alpha)^k\,.
\end{align*}
Since the above is true for all $k$, we have $\pp_{i,\rho}(\tau_j<\infty)=1$.
\end{proof}

Now we combine all the results of \Cref{subsect_tech} to prove Theorem \ref{trich}.

\begin{proof}[Proof of Theorem \ref{trich}]
Proposition \ref{12dist} shows that either $\E_{i,\rho}(n_j)=\infty $ for all $i,j$ and $\rho$, or  $\E_{i,\rho}(n_j)<\infty$ for all $i,j$ and $\rho$. \Cref{la312} combined with \Cref{ijbis} shows that in the former case, $\pp_{i,\rho}(\tau_j<\infty)=1$ for all $i,j$ and $\rho$ as well. Proposition \ref{yet} shows that, in the latter case, $\pp_{i,\rho}(\tau_j<\infty)=1$ may only occur for non-faithful $\rho$, and this concludes the proof. 
\end{proof}

\noindent \textbf{Acknowledgments}  The authors are supported by ANR project StoQ ANR-14-CE25-0003-01.

\bibliographystyle{abbrv}
\bibliography{RecTran}

\end{document}